\documentclass[11pt,a4paper,twoside]{amsart}

\usepackage[latin1]{inputenc}
\usepackage[T1]{fontenc}
\usepackage[english]{babel}
\usepackage{pifont,fancybox,pstricks,pst-grad,pst-node}
\usepackage{amssymb,amscd,amstext,latexsym,array,verbatim,txfonts,url,graphicx,array,colortbl}
\renewcommand{\le}{\leqslant}
\renewcommand{\ge}{\geqslant}

\theoremstyle{plain}
\newtheorem{thm}{\bfseries Theorem}[section]
\newtheorem{lem}[thm]{\bfseries Lemma}
\newtheorem{prop}[thm]{\bfseries Proposition}

\theoremstyle{remark}
\newtheorem{example}[thm]{\bfseries Example}

\newtheorem{rem}[thm]{\bfseries Remark}

\newcommand{\iso}{\cong}
\DeclareMathSymbol{\Z}{\mathalpha}{AMSb}{"5A} 
\DeclareMathSymbol{\PP}{\mathalpha}{AMSb}{"50} % the letter Bbb P for projective
\DeclareMathSymbol{\Q}{\mathalpha}{AMSb}{"51}
\DeclareMathSymbol{\N}{\mathalpha}{AMSb}{"4E}
\DeclareMathSymbol{\R}{\mathalpha}{AMSb}{"52}

\renewcommand{\le}{\leqslant}
\renewcommand{\ge}{\geqslant}

\newcommand{\ga}{\gamma}

\newcommand{\al}{\alpha}

\newcommand{\be}{\beta}

\newcommand{\im}{\mathrm{Im}}
\newcommand{\rank}{\mathrm{rank}}


\definecolor{mongris}{gray}{0.9}

\renewcommand{\bf}{\bfseries}
\renewcommand{\it}{\itshape}

\newcommand{\Vor}{{\mathrm Vor}}

\def \Id {{\rm Id}}
\def \wtA {{\tilde A}}
\def \wtS {{\tilde S}}
\def \Vor{{\rm Vor}}
\def \Id{{\rm Id}} %% hmm this is defined above ...!

\def \St{{S\hskip-1.5pt t}}

\newcommand{\GL}{\mathit{GL}}
\newcommand{\SL}{\mathit{SL}}

\definecolor{mongris}{rgb}{0.9, 0.9, .9}
\newcommand{\graycell}[1]{\multicolumn{1}{|>{\columncolor{mongris}}c|}{#1}}
\usepackage[all]{xy}
\renewcommand{\geq}{\geqslant}
\renewcommand{\leq}{\leqslant}

\newcommand{\cS}{\mathcal{S}}

\begin{document}
\begin{flushleft}{\bf \small }  %% Work in Progress (version coh6) - Do Not Distribute.}
\end{flushleft}
\title{Perfect forms and the cohomology of modular groups} 
\author{Philippe Elbaz-Vincent}
\address{%% Institut Fourier, Universit\'e Joseph Fourier\\
Institut Fourier, UMR 5582 (CNRS-Universit\'e Grenoble 1)\\
100 rue des Math\'ematiques\\
Domaine Universitaire\\
BP 74\\
38402 Saint Martin  d'H\`eres (France)
}
\email{Philippe.Elbaz-Vincent@ujf-grenoble.fr}
\author{Herbert Gangl}
\address{Department of Mathematical Sciences\\
South Road\\
University of Durham (United Kingdom)}
\email{herbert.gangl@durham.ac.uk}
\author{Christophe Soul\'e}
\address{IHES\\
Le Bois-Marie 35\\ 
Route de Chartres\\
91440  Bures-sur-Yvette (France)}
\email{soule@ihes.fr}
\maketitle
\tableofcontents
\markboth{{\sc Ph. Elbaz-Vincent, H. Gangl and C. Soul\'e}}{{\sc Perfect forms and the cohomology of modular groups}}
\begin{abstract} 
For $N=5$, 6 and 7, using the classification of perfect quadratic forms, 
we compute the homology of the Vorono\"{\i} cell complexes
attached to the modular groups $\SL_N(\Z)$ and $GL_N(\Z)$.
From this we deduce the rational
cohomology of those groups.
\end{abstract}
 
 \section{Introduction} Let $N\geq 1$ be an integer and let $\SL_N(\Z)$ be the modular group of integral matrices
 with determinant one. Our goal is to compute its cohomology groups with trivial coefficients, i.e. $H^q\big(\SL_N(\Z),\Z)$.
 The case $N=2$ is well-known and follows from the fact that $\SL_2(\Z)$ is the amalgamated product of two finite cyclic groups
 (\cite{Serre}, \cite{B}, II.7, Ex.3, p.52). The case $N=3$ was done in \cite{Soule-SL3}: for any $q>0$ the
 group $H^q\big(\SL_3(\Z),\Z\big)$ is killed by 12. 
 The case $N=4$ has been studied by Lee and Szczarba in \cite{LS}: modulo 2, 3 and 5--torsion, the cohomology group 
 $H^q\big(\SL_4(\Z),\Z\big)$ is trivial whenever $q>0$, except that $H^3\big(\SL_4(\Z),\Z\big) = \Z$. In Theorem 
\ref{cohomology} below, we solve the cases $N=5$, 6 and 7.

For these calculations we follow the method of \cite{LS}, i.e. we use the perfect forms of Vorono\"{\i}. Recall
from \cite{Vo} and \cite{martinet} that a perfect form in $N$ variables is a positive definite real quadratic
form $h$ on $\R^N$ which is uniquely determined (up to a scalar) by its set of integral minimal vectors.
Vorono\"{\i} proved in \cite{Vo} that there are finitely many perfect forms of rank $N$, modulo the action of $\SL_N(\Z)$.
These are known for $N\leq 8$ (see \S\ref{sec1} below).

 Vorono\"{\i} used perfect forms to define a cell decomposition of the space $X_N^*$ of positive real quadratic forms, the kernel
 of which is defined over $\Q$. This cell decomposition (cf. \S\ref{sec2}) is invariant under $\SL_N(\Z)$, hence it can be used
 to compute the equivariant homology of $X_N^*$ modulo its boundary. On the other hand, this equivariant homology
turns out to be isomorphic to the groups $H_q\big(\SL_N(\Z), \St_N\big)$, where $\St_N$ is the Steinberg module
(see \cite{BS} and \S\ref{ssec2.4} below).
Finally, Borel--Serre duality \cite{BS} asserts that the homology $H_*\big(\SL_N(\Z), \St_N\big)$ is dual to the cohomology $H^*\big(\SL_N(\Z),\Z\big)$ (modulo 
torsion).

To perform these computations for $N\leq 7$, we needed the help of a computer. The reason is that the Vorono\"{\i} cell 
decomposition of $X_N^*$ gets soon very complicated when $N$ increases. For instance, when $N=7$, there are more than
two million orbits of cells of dimension 18, modulo the action of $\SL_N(\Z)$ (see Figure~2 below). 
%%%
For this purpose, we have developed a C library \cite{PFPK}, which uses PARI \cite{PARI2} for some functionalities. The algorithms are based on exact methods.
As a result we get the full Vorono\"{i} cell decomposition of the spaces $X^*_N$ for $N\le 7$ (with either $GL_N(\Z)$ or $SL_N(\Z)$ action). Those decompositions are summarized in the figures and tables below. 
The computations were done on several computers using different processor architectures (which is useful for checking the results) and for $N=7$ the overall computational time 
was more than a year.

\smallskip
The paper is organized as follows. In \S2, we recall the Vorono\"{\i} theory of perfect forms. In \S3, we introduce a complex
of abelian groups that we call the ``Vorono\"{\i} complex'' which computes the homology groups 
$H_q\big(\SL_N(\Z),\St_N\big)$. In \S4, we explain how to get an explicit description of the Vorono\"{\i} complex in rank
$N=5$, 6 or 7, starting from the description of perfect forms available in the literature (especially in the
work of Jaquet \cite{jaquet}). In Figures~1 and 2 we display the rank of the groups in the Vorono\"{\i} complex and in Tables~1--5
we give the elementary divisors of its differentials. The homology of the Vorono\"{\i} complex (hence the groups $H_q(\SL_N(\Z),\St_N)$\,) follows from this. It is given in Theorem \ref{homology_voronoi}.

We found two methods to test whether our computations are correct. First, checking that the virtual Euler characteristic
of $\SL_N(\Z)$ vanishes leads to a mass formula for the orders of the stabilizers of the cells of $X_N^*$ (cf.~\S\ref{massformula}).
Second, the identity $d_{n-1} \circ d_n = 0$ for the differentials in the Vorono\"{\i} complex is a non-trivial equality
when these differentials are written as explicit (large) matrices.

In \S5 we give an explicit formula for the top homology group of the Vorono\"{\i} complex (Theorem \ref{explicithomclass}).
In \S6 we prove that the Vorono\"{\i} complex of $GL_N(\Z)$ is a direct factor of the Vorono\"{\i} complex of $GL_6(\Z)$
shifted by one. Finally, in \S7 we explain how to compute the cohomology of $\SL_N(\Z)$ and $GL_N(\Z)$
(modulo torsion) from our results on the homology of the Vorono\"{i} complex in \S4. Our main result is stated in Theorem
\ref{cohomology}.

\smallskip
%\medskip 
{\bf Acknowledgments:} The first two authors are particularly indebted to the IHES for its hospitality.
The second author thanks the Institute for Experimental Mathematics in Essen,
acknowledging financial support by the DFG and the European Commission as well as hospitality of the 
Newton Institute in Cambridge and the MPI for Mathematics in Bonn.
The authors are grateful to B. Allombert, J.-G.~Dumas,  D.-O.~Jaquet, J.-C. K\"onig, J.~Martinet, S.~Morita, J-P.~Serre and B.~Souvignier for helpful discussions. The computations of the Vorono\"{i} cell decomposition were performed on the computers  of the Institut de Math\'ematiques et Mod\'elisation de Montpellier, the MPI Bonn,  the Institut Fourier in Grenoble and the Centre de Calcul M\'edicis and we are grateful to those institutions.

\smallskip 
{\it Notation:}
\noindent For any positive integer $n$ we let ${\mathcal S}_n$ be the class of finite abelian groups the order of which has only prime factors less than or equal to $n$. 

\newpage

\section{Vorono{\"\i}'s reduction theory}\label{sec1}

\subsection{ Perfect forms. }\label{ssec1.1} Let $N \geq 2$ be an integer. We let $C_N$ be the set of positive definite real quadratic forms in $N$ variables. Given $h \in C_N$, let $m(h)$ be the finite set of minimal vectors of $h$, i.e. vectors $v \in {\mathbb Z}^N$, $v \ne 0$, such that $h(v)$ is minimal. A form $h$ is called {\it perfect} when $m(h)$ determines $h$ up to scalar: if $h' \in C_N$ is such that $m(h') = m(h)$, then $h'$ is proportional to $h$. 
\begin{example}
 The form $h(x,y)=x^2+y^2$ has minimum 1 and precisely 4~minimal vectors $\pm (1,0)$ and $\pm (0,1)$. This form is not perfect, because there is
 an infinite number of positive definite quadratic forms having these minimal vectors, namely the forms $h(x,y)=x^2+axy+y^2$
where $a$ is a non-negative real number less than 1.
By contrast, the form $h(x,y)=x^2+xy+y^2$ has also minimum 1 and has exactly 6 minimal vectors, viz.~the ones above and $\pm (1,-1)$. This form is perfect, the associated lattice is the ``honeycomb lattice''.
\end{example}

 Denote by $C_N^*$ the set of non-negative real quadratic forms on ${\mathbb R}^N$ the kernel of which is spanned by a proper linear subspace of ${\mathbb Q}^N$, by $X_N^*$ the quotient of $C_N^*$ by positive real homotheties, and by $\pi : C_N^* \to X_N^*$ the projection. Let $X_N = \pi (C_N)$ and $\partial X_N^* = X_N^* - X_N$. Let $\Gamma$ be either $GL_N ({\mathbb Z})$ or $\SL_N ({\mathbb Z})$. The group $\Gamma$ acts on $C_N^*$ and $X_N^*$ on the right by the formula
$$
h \cdot \gamma = \gamma^t \, h \, \gamma \, , \quad \gamma \in \Gamma \, , \ h \in C_N^* \, ,
$$
where $h$ is viewed as a symmetric matrix and $\gamma^t$ is the transpose of the matrix $\gamma$.
Vorono{\"\i} proved that there are only finitely many perfect forms modulo the action of $\Gamma$ and multiplication by positive real numbers (\cite{Vo}, Thm.~p.110).\\
The following table gives the current state of the art on the enumeration of perfect forms.\\
{\small
\[
\setlength{\arraycolsep}{3pt}
\setlength\extrarowheight{2pt}
\begin{array}{|c|c|c|c|c|c|c|c|c|c|}
\hline
 \mathrm{rank}  & 1 & 2 & 3 & 4 & 5 & 6 & 7 & 8  & 9 \\
\hline
\# \mathrm{classes} & 1 & 1 & 1 & 2 & 3 & 7& 33 & 10916 & \geqslant 500000  \\
\hline
%\mathrm{example} & \,\Z\, & \,\A_2\, & \,\A_3 \,& \,\A_4\, ,\, \D_4\, & \,\A_5\, ,\, \D_5\, , \A_5^2 \,& \,\A_6\, , \, \D_6\, ,\, \E_6\, , \, \E_6^* \,& \,  \A_7\, , \, \D_7\, ,\, \E_7\, & \, \A_8\, , \, \D_8\, ,\, \E_8\, &\, \A_9,\, \D_9 \,\\[1pt]
%\hline
%\noalign{\hrule height 0.4pt} 
\end{array}
\]
}

\medskip
The classification of perfect forms of rank 8 was achieved by Dutour, Sch\"urmann and \nobreak{Vallentin} in 2005 \cite{dsv}, \cite{Sch}.
They have also shown that in rank 9 there are at least 500000 classes of perfect forms. The corresponding classification 
for rank~7 was completed by Jaquet in 1991 \cite{jaquet}, for rank 6 by Barnes \cite{barnes}, and by Vorono\"{\i} for the other dimensions. 
We refer to the book of Martinet \cite{martinet} for more details on the results up to rank 7.

\subsection{ A cell complex}\label{ssec1.2} Given $v \in {\mathbb Z}^N - \{ 0 \}$ we let $\hat v \in C_N^*$ be the form defined by
$$
\hat v (x) = (v \mid x)^2 \, , \ x \in {\mathbb R}^N \, ,
$$
where $(v \mid x)$ is the scalar product of $v$ and $x$. The {\it convex hull in }$X_N^*$ of a finite subset $B \subset {\mathbb Z}^N - \{ \mathbf 0 \}$ is the subset of $X_N^*$ which is the image under $\pi$ of the quadratic forms $\underset{j}{\sum} \, \lambda_j \, \hat{v_j}\in C_N^*$, where $v_j \in B$ and $\lambda_j \geq 0$. For any perfect form $h$, we let $\sigma (h) \subset X_N^*$ be the convex hull of the set $m(h)$ of its minimal vectors. Vorono{\"\i} proved in \cite{Vo}, \S\S8-15, that the cells $\sigma (h)$ and their intersections, as $h$ runs over all perfect forms, define a cell decomposition of $X_N^*$, which is invariant under the action of $\Gamma$. We endow $X_N^*$ with the corresponding $CW$-topology. If $\tau$ is a closed cell in $X_N^*$ and $h$ a perfect form with $\tau \subset \sigma (h)$, we let $m(\tau)$ be the set of vectors $v$ in $m(h)$ such that $\hat v$ lies in $\tau$. Any closed cell $\tau$ is the convex hull of $m(\tau)$, and for any two closed cells $\tau$, $\tau'$ in $X_N^*$ we have  $m(\tau) \cap m(\tau') = m (\tau \cap \tau')$.

\section{The Vorono{\"\i} complex}\label{sec2}

\subsection{Definition }\label{ssec2.1} Let $d(N) = N(N+1)/2-1$ be the dimension of $X_N^*$ and $n \leq d(N)$ a natural integer. We denote by $\Sigma_n^\star= \Sigma_n^\star(\Gamma)$ a set of representatives, modulo the action of $\Gamma$, of those cells of dimension $n$ in $X_N^*$ which meet $X_N$, and by $\Sigma_n =\Sigma_n(\Gamma) \subset \Sigma_n^\star(\Gamma)$ the cells $\sigma$ for which any element of the stabilizer $\Gamma_{\sigma}$ %=\Stab_\Gamma(\sigma)
of $\sigma$ in $\Gamma$ preserves the orientation. Let $V_n$ be the free abelian group generated by $\Sigma_n$. We define as follows a map
$$
d_n : V_n \to V_{n-1} \, .
$$

For each closed cell $\sigma$ in $X_N^*$ we fix an orientation of $\sigma$, i.e. an orientation of the real vector space ${\mathbb R} (\sigma)$ of symmetric matrices spanned by the forms $\hat v$ with $v \in m(\sigma)$. Let $\sigma \in \Sigma_n$ and let $\tau'$ be a face of $\sigma$ which is equivalent under $\Gamma$ to an element in $\Sigma_{n-1}$ (i.e. $\tau'$ neither lies on the boundary nor has elements in its stabilizer reversing the orientation). Given a positive basis $B'$ of ${\mathbb R} (\tau')$ we get a basis $B$ of ${\mathbb R} (\sigma)\supset {\mathbb R} (\tau')$ by appending to $B'$ a vector $\hat v$, where $v \in m(\sigma) - m(\tau')$. We let $\varepsilon (\tau' , \sigma) = \pm 1$ be the sign of the orientation of $B$ in the oriented vector space ${\mathbb R} (\sigma)$ (this sign does not depend on the choice of $v$).

\smallskip

Next, let $\tau \in \Sigma_{n-1}$ be the (unique) cell equivalent to $\tau'$ and let $\gamma\in\Gamma$ be such that $\tau'= \tau \cdot \gamma$. We define $\eta (\tau , \tau') = 1$ (resp. $\eta (\tau , \tau') = -1$) when $\gamma$ is compatible (resp. incompatible) with the chosen orientations of ${\mathbb R} (\tau)$ and ${\mathbb R} (\tau')$.

\smallskip

Finally we define
\begin{equation}
\label{eq1}
d_n (\sigma) = \sum_{\tau \in \Sigma_{n-1}} \sum_{\tau'} \eta (\tau , \tau') \, \varepsilon (\tau' , \sigma) \, \tau \, ,
\end{equation}
where $\tau'$ runs through the set of faces of $\sigma$ which are equivalent to $\tau$.

\subsection{ A spectral sequence }\label{ssec2.2} According to \cite{B}, VII.7, there is a spectral sequence $E_{pq}^r$ converging to the equivariant homology groups $H_{p+q}^{\Gamma} (X_N^* , \partial X_N^* ; {\mathbb Z})$ of the homology pair $(X_N^* , \partial X_N^*)$, and such that 
$$
E_{pq}^1 = \bigoplus_{\sigma \in \Sigma_p^\star} H_q (\Gamma_{\sigma} , {\mathbb Z}_{\sigma}) \, ,
$$
where ${\mathbb Z}_{\sigma}$ is the orientation module of the cell $\sigma$ and, as above, $\Sigma_p^\star$ is a set of representatives, modulo $\Gamma$, of the $p$-cells $\sigma$ in $X_N^*$ which meet $X_N$. Since $\sigma$ meets $X_N$, its stabilizer $\Gamma_{\sigma}$ is finite and, by Lemma \ref{lemma1} in \S\ref{sec4} below, the order of $\Gamma_{\sigma}$ is divisible only by primes $p \leq N+1$. Therefore, when $q$ is positive, the group $H_q (\Gamma_{\sigma} , {\mathbb Z}_{\sigma})$ lies in ${\mathcal S}_{N+1}$. 

When $\Gamma_{\sigma}$ happens to contain an element which changes the orientation of $\sigma$, the group $H_0 (\Gamma_{\sigma} , {\mathbb Z}_{\sigma})$ is killed by $2$, otherwise  $H_0 (\Gamma_{\sigma} , {\mathbb Z}_{\sigma})\cong {\mathbb Z}_{\sigma})$.
%% \bra otherwise 
%% $\Gamma_{\sigma} $ acts trivially on ${\mathbb Z}_{\sigma})$ and $H_0 (\Gamma_{\sigma} , {\mathbb Z}_{\sigma})\cong {\mathbb Z}_{\sigma})$??\ket. 
Therefore, modulo ${\mathcal S}_2$, we have
$$
E_{n\,0}^1 = \bigoplus_{\sigma \in \Sigma_n} {\mathbb Z}_{\sigma} \, ,
$$
and the choice of an orientation for each cell $\sigma$ gives an isomorphism between $E_{n\,0}^1$ and $V_n$.

\subsection{Comparison} We claim that the differential
$$
d_n^1 : E_{n\,0}^1 \to E_{n-1,0}^1
$$
coincides, %%\bra modulo ${\mathcal S}_2$??\ket, 
up to sign, with the map $d_n$ defined in \ref{ssec2.1}. According to \cite{B}, VII, Prop.~(8.1), the differential $d_n^1$ can be described as follows.

\smallskip
%% \bra For the following paragraph we only seem to care about $\sigma \in \Sigma_n$ and $\tau' \in \Sigma_{n-1}$; no need to consider $\tau'$ a face of $\sigma$, which is good as it makes things simpler.\ket

\smallskip
Let $\sigma \in \Sigma_n^\star$ and let $\tau'$ be a face of $\sigma$. Consider the group $\Gamma_{\sigma \tau'} = \Gamma_{\sigma} \cap \Gamma_{\tau'}$ and denote by
$$
t_{\sigma \tau'} : H_* (\Gamma_{\sigma} , {\mathbb Z}_{\sigma}) \to H_* (\Gamma_{\sigma \tau'} , {\mathbb Z}_{\sigma})
$$
the transfer map. Next, let
$$
u_{\sigma \tau'} : H_* (\Gamma_{\sigma\tau'} , {\mathbb Z}_{\sigma}) \to H_* (\Gamma_{\tau'} , {\mathbb Z}_{\tau'})
$$
be the map induced by the natural map %% differential $\partial_{\sigma\tau'} 
${\mathbb Z}_{\sigma} \to {\mathbb Z}_{\tau'}$, together with the inclusion $\Gamma_{\sigma\tau'} \subset \Gamma_{\tau'}$. Finally, let $\tau \in \Sigma_{n-1}^\star$ be the representative of the $\Gamma$-orbit of $\tau'$, let $\gamma \in \Gamma$ be such that $\tau' = \tau \cdot \gamma$, and let
$$
v_{\tau'\tau} : H_* (\Gamma_{\tau'} , {\mathbb Z}_{\tau'}) \to H_* (\Gamma_{\tau} , {\mathbb Z}_{\tau})
$$
%% \bra This map also depends on $\tau'$: there could actually be several faces of $\sigma$ equivalent to $\tau$. Suggestion: call the map $v_{\tau\tau'}$ instead.\ket

\smallskip\noindent
be the isomorphism induced by $\gamma$. Then the restriction of $d_n^1$ to $H_* (\Gamma_{\sigma} , {\mathbb Z}_{\sigma})$ is equal, up to sign, to the sum
\begin{equation}
\label{eq2}
\sum_{\tau'} v_{\tau'\tau} \, u_{\sigma \tau'} \, t_{\sigma \tau'} \, ,
\end{equation}
where $\tau'$ runs over a set of representatives of faces of $\sigma$ modulo $\Gamma_{\sigma}$.

\smallskip

To compare $d_n^1$ with $d_n$ we first note that, when $\tau \in \Sigma_{n-1}$, 
$$
v_{\tau'\tau} : H_0 (\Gamma_{\tau'} , {\mathbb Z}_{\tau'}) = {\mathbb Z} \to H_0 (\Gamma_{\tau} , {\mathbb Z}_{\tau}) = {\mathbb Z}
$$
is the multiplication by $\eta (\tau,\tau')$, as defined in \S\ref{ssec2.1}. Next, when $\sigma \in \Sigma_n$, the map
$$
u_{\sigma\tau'} : H_0 (\Gamma_{\sigma\tau'} , {\mathbb Z}_{\sigma}) = {\mathbb Z}_{\sigma} = {\mathbb Z} \to H_0 (\Gamma_{\tau'} , {\mathbb Z}_{\tau'}) = {\mathbb Z}
$$
is the multiplication by $\varepsilon (\tau',\sigma)$, up to a sign depending on $n$ only. Finally, the transfer map
$$
t_{\sigma\tau'} : H_0 (\Gamma_{\sigma} , {\mathbb Z}_{\sigma}) = {\mathbb Z} \to H_0 (\Gamma_{\sigma\tau'} , {\mathbb Z}_{\sigma}) = {\mathbb Z}
$$
is the multiplication by $[\Gamma_{\sigma} : \Gamma_{\sigma\tau'}]$. Multiplying the sum (\ref{eq2}) by this number amounts  to the same as taking the sum over all faces of $\sigma$ as in (\ref{eq1}). This proves that $d_n$ coincides, up to sign, with $d_n^1$ on $E_{n\,0}^1 = V_n$. \qed

In particular, we get that $d_{n-1} \circ d_n = 0$.  Note that this identity will give us a non-trivial test of our explicit 
computations of the  complex.

%% \bra We need to introduce a name for the Vorono\"{\i} complex, suggestion $Vor_\Gamma$, as $V_N$
%% already has a meaning and we also want to distinguish between SL and GL.\ket

\smallskip
{\it Notation:}
\noindent The resulting complex  $(V_\bullet,d_\bullet)$ will be denoted by $\Vor_\Gamma$, and we call it the {\em Vorono\"i complex}.

\subsection{The Steinberg module }\label{ssec2.4} Let $T_N$ be the spherical Tits building of $\SL_N$ over ${\mathbb Q}$, i.e. the simplicial set defined by the ordered set of non-zero proper linear subspaces of ${\mathbb Q}^N$. The reduced homology $\tilde H_q (T_N,\Z)$ of $T_N$ with integral coefficients is zero except when $q = N-2$, in which case
$$
\tilde H_{N-2} (T_N,\Z) = {\rm St}_N
$$
is by definition the Steinberg module \cite{BS}. According to \cite{SouleK4}, Prop.~1, the relative homology groups $H_q (X_N^* , \partial X_N^* ; {\mathbb Z})$ are zero except when $q = N-1$, and
$$
H_{N-1} (X_N^* , \partial X_N^* ; {\mathbb Z}) = {\rm St}_N \, .
$$
From this it follows that, for all $m \in {\mathbb N}$,
$$
H_m^{\Gamma} (X_N^* , \partial X_N^* ; {\mathbb Z}) = H_{m-N+1} (\Gamma , {\rm St}_N)
$$
(see e.g. \cite{SouleK4}, \S3.1). Combining this equality with the previous sections we conclude that, modulo ${\mathcal S}_{N+1}$,
\begin{equation}
\label{eq3}
H_{m-N+1} (\Gamma , {\rm St}_N) = H_m (\Vor_\Gamma) \, .
\end{equation}
\section{The Vorono\"{\i} complex in dimensions 5, 6 and 7}
\noindent In this section, we explain how to compute the Vorono\"{\i} complexes of rank $N\leq 7$.
\subsection{Checking the equivalence of cells}\label{equiv_cells}
As a preliminary step, we develop an effective method to check whether two cells $\sigma$ and $\sigma'$ of the same dimension are equivalent under the
action of $\Gamma$. The cell $\sigma$ (resp. $\sigma'$) is described by its set of minimal vectors $m(\sigma)$ (resp. $m(\sigma')$).
We let $b$ (resp. $b'$) be the sum of the forms $\hat v$ with $v\in m(\sigma)$ (resp. $m(\sigma')$).
If $\sigma$ and $\sigma'$ are equivalent under the action of $\Gamma$ the same is true for $b$ and $b'$,
and the converse holds true since two cells of the same dimension are equal when they have an interior point in common. 

To compare $b$ and $b'$ we first check whether or not they have the same  determinant. In case they do,
we let $M$ (resp. $M'$) be the set of numbers $b(x)$ with $ x\in m(\sigma)$ (resp.
$b'(x)$ with $x\in m(\sigma')$). If  $b$ and $b'$ are equivalent, then the sets $M$ and $M'$
must be equal. 

Finally, if $M=M'$ we check if $b$ and $b'$ are equivalent by applying an algorithm of Plesken and Souvignier \cite{souvignier}
(based on an implementation of Souvignier).

\subsection{Finding generators of the Vorono\"{\i} complex} 
In order to compute $\Sigma_n$ (and $\Sigma_n^\star$), we proceed as follows. Fix $N\leq 7$.
Let ${\mathcal P}$ be a set of representatives of the perfect forms of rank $N$.
A choice of ${\mathcal P}$ is provided by Jaquet \cite{jaquet}.
Furthermore, for each $h\in{\mathcal P}$, Jaquet  gives the list $m(h)$ of its
minimal vectors, and the list of all perfect forms $h'\gamma$ (one for each orbit under 
$\Gamma_{\sigma(h)}$), where $h'\in{\mathcal P}$ and $\gamma\in \Gamma$,
such that $\sigma(h)$ and $\sigma(h')\gamma$ share a face of
codimension one. 
%%%%%%%%%
This provides a complete list ${\mathcal C}_{h}^1$ of representatives
of codimension one faces in $\sigma(h)$. 

From this, one deduces the full list ${\mathcal F}_{h}^1$ of faces of codimension one
in $\sigma(h)$ as follows: first list all the elements in the automorphism group 
$\Gamma_{\sigma(h)}$; this can be obtained by using a second procedure implemented by
Souvignier \cite{souvignier} which gives generators for $\Gamma_{\sigma(h)}$. We represent the latter
generators as elements in the symmetric group ${\mathfrak S}_M$, where $M$ is the cardinality of
$m(h)$, 
acting on set $m(h)$ of minimal vectors. Using those generators, we let GAP \cite{GAP} list all the elements
of $\Gamma_{\sigma(h)}$, viewed as elements of the symmetric group above.

The next step is to create a shortlist ${\mathcal F}_{h}^2$ of codimension~2 facets of 
$\sigma(h)$ by intersecting
all the translates under ${\mathfrak S}_M$ of codimension~1 facets with 
each member of ${\mathcal C}_{h}^1$ and only keeping those intersections with the correct
rank (=$d(N)-2$). The resulting shortlist is reasonably small and we apply 
the procedure of \ref{equiv_cells} to reduce the shortlist to a set of representatives 
${\mathcal C}_{h}^2$  of codimension~2 facets.

\smallskip We then proceed by induction on the codimension to define a list ${\mathcal F}_{h}^p$
of cells of codimension $p>2$ in $\sigma(h)$. Given ${\mathcal F}_{h}^p
$, we let ${\mathcal C}_{h}^p \subset {\mathcal F}_{h}^p$ be a set of representatives 
for the action of $\Gamma$. We then let ${\mathcal F}_{h}^{p+1}$ be the set of cells
$\varphi \cap \tau$, with $\varphi \in {\mathcal F}_{h}^2$, and $\tau \in {\mathcal C}_{h}^p$.

\smallskip
Next, we let $\Sigma_n^\star$ be a system of representatives modulo $\Gamma$ 
in the union of the sets ${\mathcal C}_{h}^{d(N)-n}, h\in {\mathcal P}$.
We then compute generators of the stabilizer of each cell in $\Sigma_n^\star$
with the help of another algorithm developed by Plesken and Souvignier in \cite{souvignier},
and we check whether all generators preserve the orientation of the cell.
This gives us the set $\Sigma_n$  as the set of those cells which
pass that check.

\quad\\[5pt]
\begin{figure*}[t]
{\fontsize{7}{8}
\[
\setlength{\arraycolsep}{5pt}
\setlength\extrarowheight{2pt}
\begin{array}{|c|c|c|c|c|c|c|c|c|c|c|c|c|c|c|c|c|c|}
\hline
\graycell{\mathbf{n}}  & \graycell{4}  & \graycell{5} & \graycell{6} & \graycell{7} & \graycell{8} & \graycell{9} & \graycell{10} & \graycell{11} & \graycell{12} & \graycell{13} & 
\graycell{14} & \graycell{15} & \graycell{16} & \graycell{17} & \graycell{18} &\graycell{19} & \graycell{20}  \\
\hline
\Sigma_n^\star\big( GL_5(\Z)\big) & 2 & 5& 10& 16& 23& 25& 23& 16& 9& 4& 3 & & & &&& \\
\hline
\Sigma_n\, \big(GL_5(\Z)\big) &  & & &  & 1 & 7 & 6 & 1& 0 & 2 & 3  & & & & & & \\[1pt]
\hline
\Sigma_n^\star\big(GL_6(\Z)\big) &   &  3 & 10& 28& 71& 162& 329& 589&  874& 1066& 1039& 775& 425& 181& 57& 18& 7 \\
\hline
\Sigma_n\,\big( GL_6(\Z)\big)&  &  & & &  & 3 & 46 & 163 & 340 & 544 & 636 & 469 & 200 & 49 & 5 &  & \\[1pt]
\hline
\Sigma_n^\star\big(\SL_6(\Z) \big)&  &  3 & 10 & 28& 71 & 163 & 347 & 691 & 1152 & 1532 & 1551 
& 1134 & 585 & 222 & 62 & 18 & 7\\
\hline
\Sigma_n\,\big(\SL_6(\Z)\big) &  &  & 3 & 10 & 18 & 43 & 169 & 460 & 815 & 1132 & 1270 
& 970 & 434 & 114 & 27 & 14 & 7\\[1pt]
\noalign{\hrule height 0.4pt} 

\end{array}
\]
}\caption{Cardinality of {$\Sigma_n$} and  {$\Sigma_n^\star$} for $N=5,6$ (empty
slots denote zero).}\label{fig:sig56}
\end{figure*}

\begin{figure*}[t]
{\fontsize{7}{8}
\[
\setlength{\arraycolsep}{5pt}
\setlength{\extrarowheight}{4pt}
\begin{array}{|c|c|c|c|c|c|c|c|c|c|c|c|}
\hline
\graycell{\mathbf{n}} &  \graycell{6} & \graycell{7} & \graycell{8} & \graycell{9} & \graycell{10} & \graycell{11} & \graycell{12 } & \graycell{13}& \graycell{14} & \graycell{15} &  \graycell{16}\\
\hline
\Sigma_n^\star &{6}& {28} & {115} & {467} & {1882} & {7375} & {26885} & {87400}& {244029} & {569568}&  {{1089356} }\\
\hline
\Sigma_n &  & & & {1} & {60} & {1019} & {8899}& {47271}& {171375} & {460261}&  {{955128}} \\
\hline
\hline
\graycell{\mathbf{n}}   & \graycell{17} & \graycell{18}& \graycell{19}& \graycell{20}& \graycell{21}& \graycell{22 }& \graycell{23} & \graycell{24 }& \graycell{25 }& \graycell{26} & \graycell{27}\\ 
\hline
\Sigma_n^\star & {{1683368} }&  {{2075982} }& {{2017914} }&{{1523376}}& {{876385}}&  {374826}&
 {115411}& {24623}& {3518}& {352}& 33\\
\hline
\Sigma_n & {{1548650}} &  {{1955309}} & {{1911130}}&{{1437547}}& {{822922}} & {349443}& 
 {105054}& {21074}& {2798}& {305}& 33\\[1pt]
\noalign{\hrule height 0.4pt} 
\end{array}
\]
}\caption{Cardinality of {$\Sigma_n$} and  {$\Sigma_n^\star$} for $GL_7(\Z)$ .}\label{fig:sig7}
\end{figure*}

\begin{prop}\label{cardinality} 
The cardinality of {$\Sigma_n$} and {$\Sigma_n^\star$} is displayed in Figure
\ref{fig:sig56} for rank $N=5,6$ and  in Figure \ref{fig:sig7} for rank $N=7$.
\end{prop}

\begin{rem}
The first line in Figure \ref{fig:sig56} has already been computed by Batut (cf. \cite{batut}, p.409, second column of Table~2).
\end{rem}

\subsection{ The differential }\label{ssec3.4}

The next step is to compute the differentials of the  Vorono\"{\i} complex
by using formula \eqref{eq1} above.
In Table \ref{tab:gl6}, we give information on the differentials in the Vorono\"{\i}
complex of rank 5. For instance the second line, denoted $d_{11}$, is about the differential from
$V_{11}$ to $V_{10}$. In the bases $\Sigma_{11}$ and $\Sigma_{10}$, this differential is given by a matrix
$A$ with $\Omega=513$ non-zero entries, with $m=46= \rm{card}(\Sigma_{10})$ rows and $n=163=\rm{card}(\Sigma_{11})$
columns. The rank of $A$ is 42, and the rank of its kernel is 121. The elementary divisors of $A$ are $1$ (multiplicity $40$) and $2$ (multiplicity $2$).

The cases of $\SL_4(\Z)$,  $GL_6(\Z)$ and  $\SL_6(\Z)$ are treated in Table  \ref{tab:sl4},  Table  \ref{tab:gl6} and  Table \ref{tab:sl6}, respectively. \\

Our results on the  differentials in rank $7$ are
shown in Table \ref{tab:gl7}. While the matrices are sparse, they are not sparse enough for efficient computation. They have a poor conditioning with some dense columns or rows (this is a consequence of the fact that the complex is not simplicial and non-simplicial cells can have a large number of non-trivial intersections with the faces).
We have obtained  full information on
the rank of the differentials. For the computation of the elementary divisors 
complete results have been obtained in the case of matrices of $d_n$ for %$n=10$, 11, 12, 13, 14, 24, 25, 26 and 27 
$10\leq n\leq 14$ and $24\leq n\leq 27$
only.  See \cite{pasco2007} for a detailed description of the computation.\\

\begin{table*}\footnotesize
\begin{center}
\begin{tabular}{|r|r|r|r|r|r|r|}
\hline
\graycell{$A$} & \graycell{$\Omega$} & \graycell{$n$} & \graycell{$m$} & \graycell{rank} & \graycell{ker} & \graycell{elementary divisors} \\
\hline
\hline
$d_4$& 0 & 1& 0   &0  &1 & \\\hline
$d_5$& 1&  1 &1   &1  &0 & 1(1) \\\hline
$d_6$&0 &  1 &1   &0  &1 & \\\hline
$d_7$& 0& 0&  1&  0   &0 & \\\hline
$d_8$& 0& 1&  0&  0   &1 & \\\hline
$d_9$& 2& 2&  1&  1   &1 & 2(1)\\\hline
%\\\hline                                                                                           
\end{tabular}
\medskip
\caption{Results on the rank and elementary divisors of the differentials
 for $\SL_4(\Z)$.} 
 \label{tab:sl4}
\end{center}
\end{table*}

\begin{table*}\footnotesize
\begin{center}
\begin{tabular}{|r|r|r|r|r|r|r|}
\hline 
\graycell{$A$} & \graycell{$\Omega$} & \graycell{$n$} & \graycell{$m$} & \graycell{rank} & \graycell{ker} & \graycell{elementary divisors} \\
\hline
\hline
$d_8$& 0 & 1& 0   &0  &1 & \\\hline
$d_9$& 2&  7 &1   &1  &6 & 1(1)\\\hline
$d_{10}$&18 &6 &7   &5  &1 & 1(4), 2(1)\\\hline
$d_{11}$& 5& 1& 6&  1   &0 & 1(1) \\\hline
$d_{12}$& 0& 0& 1&  0   &0 & \\\hline
$d_{13}$& 0& 2& 0&  0   &2 & \\\hline
$d_{14}$& 4& 3& 2&  2   &1 & 5(1), 15(1)\\\hline
%\\\hline
\end{tabular}
\medskip
\caption{Results on the rank and elementary divisors of the differentials 
 for $GL_5(\Z)\,$.} 
\label{tab:gl5}
\end{center}
\end{table*}

\begin{table*}\footnotesize
\begin{center}
\begin{tabular}{|r|r|r|r|r|r|r|}
\hline 
\graycell{$A$} & \graycell{$\Omega$} & \graycell{$n$} & \graycell{$m$} & \graycell{rank} & \graycell{ker} & \graycell{elementary divisors} \\
\hline
\hline
$d_{10}$&17 &46& 3& 3 &43& 1(3)\\\hline
$d_{11}$& 513&163& 46& 42 &121& 1(40), 2(2)\\\hline
$d_{12}$&  2053&340& 163& 120& 220& 1(120)\\\hline
$d_{13}$& 4349&544 &340 &220 &324 &1(217), 2(3)\\\hline
$d_{14}$&  6153&636 &544 &324 &312& 1(320), 2(1), 6(2), 12(1)\\\hline
$d_{15}$& 5378&469 &636 &312 &157& 1(307), 2(3), 60(2) \\\hline
$d_{16}$&  2526&200 &469 &156 &44& 1(156)\\\hline
$d_{17}$& 597&49 &200 &44 &5 &1(41), 3(1), 6(1), 36(1)\\\hline
$d_{18}$& 43&5 &49 &5 &0&  1(5)\\\hline
%\\\hline
\end{tabular}
\medskip
\caption{Results on the rank and elementary divisors of the differentials for $GL_6(\Z)\,$.}
\label{tab:gl6}
\end{center}
\end{table*}
\begin{table*}\footnotesize
\begin{center}
\begin{tabular}{|r|r|r|r|r|r|r|}
\hline 
\graycell{$A$} & \graycell{$\Omega$} & \graycell{$n$} & \graycell{$m$} & \graycell{rank} & \graycell{ker} & \graycell{elementary divisors} \\
\hline
\hline
$d_7$& 12&10&3  &3 &7& 1(3) \\\hline
$d_8$& 48&18&10  &7 & 11& 1(7)\\\hline
$d_9$&  140&43&18  &11 & 32 &1(11)\\\hline
$d_{10}$&  613&169 &43 &32 &  137 &  1(32)\\\hline
$d_{11}$&  2952&460 &169 &136 & 324& 1(129), 2(6), 6(1)\\\hline
$d_{12}$& 7614&815 &460 &323 &  492& 1(318), 2(3), 4(2)\\\hline
$d_{13}$&  12395&1132 &815 &491 &  641& 1(491)\\\hline
$d_{14}$& 14966&1270 &1132 &641 &  629& 1(637), 3(3), 12(1)\\\hline
$d_{15}$&  12714&970 &1270 &629 & 341 & 1(621), 2(5), 6(1), 60(2)\\\hline
$d_{16}$&  6491&434 &970 &339 & 95 & 1(338), 2(1)\\\hline
$d_{17}$&  1832&114 &434 &95 & 19 &  1(92), 3(2), 18(1)\\\hline
$d_{18}$&  257&27 &114 &19 & 8& 1(17), 2(2)\\ \hline
$d_{19}$&  62&14 &27 &8 & 6& 1(7), 10(1)\\\hline
$d_{20}$&  28& 7& 14 & 6& 1 & 1(1), 3(4), 504(1)\\\hline
%\\\hline
\end{tabular}
\medskip
\caption{Results on the rank and elementary divisors of the differentials for $\SL_6(\Z)\,$.}\label{tab:sl6}
\end{center}
\end{table*}

\begin{table*}\footnotesize
\begin{center}
\begin{tabular}{|r|r|r|r|r|r|r|}
\hline 
\graycell{$A$} & \graycell{$\Omega$} & \graycell{$n$} & \graycell{$m$} & \graycell{rank} & \graycell{ker} & \graycell{elementary divisors} \\
\hline
\hline
$d_{10}$& 8& 60& 1& 1& 59& 1 \\\hline
$d_{11}$& 1513&1019& 60& 59& 960& 1 (59) \\\hline
$d_{12}$& 37519&8899& 1019& 960& 7939& 1 (958), 2 (2)\\\hline
$d_{13}$& 356232&47271& 8899& 7938& 39333& 1 (7937), 2 (1)\\\hline
$d_{14}$& 1831183&171375& 47271& 39332& 132043& 1 (39300), 2 (29), 4 (3)\\\hline
$d_{15}$& 6080381&460261& 171375& 132043& 328218&    ? 
\\\hline
$d_{16}$& 14488881&955128& 460261& 328218& 626910& ? \\\hline
$d_{17}$& 25978098&1548650& 955128& 626910& 921740& ? \\\hline
$d_{18}$& 35590540&1955309& 1548650&921740 &1033569 & ? \\\hline
$d_{19}$& 37322725&1911130& 1955309&1033568 &877562 & ? \\\hline
$d_{20}$& 29893084&1437547& 1911130& 877562& 559985& ? \\\hline
$d_{21}$& 18174775&822922& 1437547& 559985& 262937& ?  \\\hline
$d_{22}$& 8251000&349443& 822922& 262937& 86506& ?\\\hline
$d_{23}$& 2695430&105054& 349443& 86505& 18549& ?
\\\hline
$d_{24}$& 593892& 21074& 105054& 18549& 2525& 1 (18544), 2 (4), 4 (1)\\\hline
$d_{25}$& 81671&2798& 21074& 2525& 273& 1 (2507), 2 (18)\\\hline
$d_{26}$& 7412&305& 2798& 273& 32& 1 (258), 2 (7), 6 (7), 36 (1)\\\hline
$d_{27}$& 600& 33&305&32&1& 1 (23), 2 (4),  28 (3),  168 (1), 2016 (1)
\\\hline
\end{tabular}
\medskip
\caption{Results on the rank and elementary divisors of the differentials for $GL_7(\Z)\,$.
}\label{tab:gl7}
\end{center}
\end{table*}
\subsection{The homology of the Vorono\"{i} complexes}
 From the computation of the differentials,  we can  determine the  homology of Vorono\"{\i} complex. Recall that if we have a 
complex of free abelian groups
\[
\cdots \to \Z^\al \overset{f}{\to} \Z^\be \overset{g}{\to} \Z^\ga \to \cdots 
\]
with $f$ and $g$ represented by matrices, then
the homology is  
\[
           \ker (g) / \im (f ) \iso \Z/d_1 \Z \oplus \cdots  \oplus \Z/d_\ell \Z \oplus \Z^{\be-\rank(f)-\rank(g)} ,
\]
where $d_1,\dots,d_\ell$ are the elementary divisors of the matrix of $f$.

We deduce from Tables 1--5 the following result on the homology of the Vorono\"{\i} complex.
\begin{thm}\label{homology_voronoi} 
 The non-trivial homology of the Vorono\"{\i} complexes associated to $GL_N(\Z)$ with $N=5,6$ modulo $\mathcal{S}_{5}$ is given by:
\begin{align*}
 H_n(\Vor_{GL_5(\Z)}) &\iso \Z, \quad \text{ if \ }n=9, 14\,,\\
 H_n(\Vor_{GL_6(\Z)}) &\iso \Z, \quad \text{ if \ }n=10,11,15\, ,
 \end{align*}
while in the case
$\SL_6(\Z)$ we get, modulo $\cS_7$, that
\begin{align*}
 H_n(\Vor_{\SL_6(\Z)}) &\iso   \begin{cases} \Z,\quad \text{ if \ }n=10, 11, 12, 20\, ,\\                          \Z^2,\quad \text{ if \ }n=15\, .
                            \end{cases}
 \end{align*}
Furthermore, for $N=7$ we get 
\begin{equation*}
H_n(\Vor_{GL_7(\Z)}\otimes \Q) \iso   \begin{cases}  \Q & \text{ if \ }n=12,13,18,22,27\,,\\ 0& \text{otherwise}.\end{cases}
\end{equation*}

\end{thm}
Notice that, if $N$ is odd, $\SL_N(\Z)$ and $GL_N(\Z)$ have the same homology modulo ${\mathcal S}_2$.
Notice also that, for simplicity, in the statement of the theorem we did not use the full information given by the list of 
elementary divisors in Tables 1--5.

\subsection{Mass formulae for  the Vorono\"{\i} complex}\label{massformula}
Let $\chi(\SL_N(\Z))$ be the virtual Euler characteristic of
the group $\SL_N(\Z)$. 
It can be computed in two ways.
First, the mass formula in \cite{B} gives
{
\[  
\chi(\SL_N(\Z))=\sum_{\sigma\in E} (-1)^{{\rm dim}(\sigma)} \frac{1}{|\Gamma_\sigma|} = \sum_{n=N}^{d(N)} (-1)^n 
\sum_{\sigma\in \Sigma_n^\star}\frac 1{|\Gamma_\sigma|}\, ,
\]}\noindent
where $E$ is a family of representatives of the cells of the Vorono\"{\i} complex 
 of rank $N$ modulo the action of 
$\SL_N(\Z)$, and $\Gamma_\sigma$ is the stabilizer of 
$\sigma$ in $\SL_N(\Z)$.
Second, by a result of Harder \cite{harder}, we know that 
{
\[
\chi(\SL_N(\Z))=\prod_{k=2}^N \zeta(1-k)\, , 
\]}\noindent
hence $\chi(\SL_N(\Z))=0$ if $N\ge 3$.\\

A non-trivial check of our computations is to test the compatibility of  these two
formulas, and the corresponding check for rank $N=5$ had been performed by
Batut (cf.~\cite{batut}, where a proof of an analogous statement, for any $N$, but instead pertaining
to {\em well-rounded} forms, which in our case are precisely the ones in $\Sigma_\bullet^\star$, is attributed to Bavard \cite{Bavard}).

If we add together the terms $ \frac{1}{|\Gamma_\sigma|} $ for cells $\sigma$ of the same dimension to a single term,
then we get for $N=6$, starting with the top dimension, 

 {\small
\begin{gather*}
\frac{45047}{1451520}
-\frac{10633}{11520}
+\frac{6425}{576}
-\frac{12541}{192}\\
+\frac{7438673}{34560}
-\frac{3841271}{8640}
+\frac{9238}{15}
-\frac{266865}{448}
+\frac{14205227}{34560}
-\frac{14081573}{69120}\\
+\frac{830183}{11520}
-\frac{205189}{11520}
+\frac{61213}{20736}
-\frac{1169}{3840}
+\frac{17}{1008}
-\frac{1}{2880}\\
=\chi(\SL_6(\Z))=0\, .\phantom{\big|}
\end{gather*}
}\noindent

For $N=7$ we obtain similarly

{\smallskip	
\begin{gather*}
-\dfrac{290879}{107520}+\dfrac{13994381}{103680}-\dfrac{31815503}{13824}+\dfrac{1362329683}{69120}-\dfrac{6986939119}{69120}\\
+\dfrac{7902421301}{23040}-\dfrac{340039739981}{414720}+\dfrac{174175928729}{120960}-\dfrac{132108094091}{69120}\\
+\dfrac{27016703389}{13824}-\dfrac{13463035571}{8640}+\dfrac{14977461287}{15360}-\dfrac{22103821919}{46080}\\
+\dfrac{8522164169}{46080}-\dfrac{17886026827}{322560}+\dfrac{1764066533}{138240}-\dfrac{101908213}{46080}+\dfrac{12961451}{46080}\\
-\dfrac{10538393}{414720}+\dfrac{162617}{103680}-\dfrac{721}{11520}+\dfrac{43}{32256}\\
=\chi\left(\SL_7(\mathbb{Z})\right)=0\,.\phantom{\big|}
\end{gather*}
}\noindent

%\medskip
\section{Explicit homology classes} 
\subsection{Equivariant fundamental classes}
\begin{thm}\label{explicithomclass}
The top homology group $H_{d(N)}\big(\Vor_{\SL_N(\Z)}\otimes \Q\big)$
has dimension~1. When $N=4$, $5$, $6\,$ or $\,7$, it is represented by the cycle

$$\sum_\sigma \frac{1}{|\Gamma_\sigma|} [\sigma]\,,$$
where $\sigma$ runs through the perfect forms of rank $N$ and the orientation of each cell
is inherited from the one of $X_N/\Gamma$.
\end{thm}
\begin{proof}
The first assertion is clear since, by \eqref{eq3} above and \eqref{BorelSerre} below we have
$$H_{d(N)}\big(\Vor_{\SL_N(\Z)}\otimes \Q\big) \iso H_{d(N)-N+1}\big(\SL_N(\Z), \St_N\otimes \Q\big) \iso H^0(\SL_N(\Z),\Q)\iso \Q\,.$$

In order to prove the second claim, write the differential between codimension~0 and codimension~1 cells as a matrix $A$ of size
$n_1\times n_0$, with $n_i=|\Sigma_{d(N)-i}(\Gamma)|$ denoting the number of codimension $i$ cells in the
Vorono\"{\i} cell complex. It can be checked that in each of the $n_1$ rows of $A$ there are precisely two non-zero
entries. Moreover, the absolute value of the $(i,j)$-th entry of $A$ is equal to the quotient $|\Gamma_{\sigma_j}|/|\Gamma_{\tau_i}|$ (an {\em integer}), where $\sigma_j\in \Sigma_0(\Gamma)$ and $\tau_i\in \Sigma_1(\Gamma)$.
Finally, one can multiply some columns by $-1$ (which amounts to changing the orientation of the corresponding codimension 0 cell) in such a way that each row has exactly one positive
and one negative entry. \end{proof}

\begin{example} For $N=5$ the differential
matrix $d_{14}$ (cf. Table \ref{tab:gl5}) between codimension~0 and codimension~1 is given by 
$$\left(\begin{matrix} 40&0&-15\\40&-15&0 \end{matrix}\right)\,,$$
so the kernel is generated by $(3,8,8)=11520\,\big(\frac1{3840},\frac{1}{1440},\frac{1}{1440}\big)$,
while the orders of the three automorphism groups are $3840$, 1440 and 1440, respectively. 
\end{example}

\begin{example}
Similarly, the differential $d_{20}:V_{20}\to V_{19}$ for rank $N=6$ (cf.~Table \ref{tab:gl6}) is 
represented by the matrix
$$\left(\begin{matrix} 
0& 0& 96& 0& 0& 0& -21\\
 3240& 0& 0& 0& -21& 0& 0\\
  0& 0& 1440& 0& 0& -3& 0 \\
  0& 0& 0& 18& 0& -6& 0 \\
  -12960& 0& 0& 0& 0& 12& 0 \\
  -3240& 0& 0& 9& 0& 0& 0\\
   0& -360& 0& 1& 0& 0& 0\\
    -4320& 0& 0& 12& 0& 0& 0\\
     0& 0& 960& -6& 0& 0& 0 \\
     0& -216& 96& 0& 0& 0& 0 \\
     -45& 45& 0& 0& 0& 0& 0\\
 -2592& 0& 1152& 0& 0& 0& 0 \\
 -3240& 0& 1440& 0& 0& 0& 0 \\
 -432& 0& 192& 0& 0& 0& 0
 \end{matrix}\right)\,.$$
Its kernel is generated by
$$(28,28,63,10080,4320,30240,288)$$
while the orders of the corresponding automorphism groups are, respectively, 
$$103680, 103680, 46080, 288, 672, 96, 10080\,,$$
and we note that $28\cdot 103680 = 63\cdot 46080 = 10080\cdot  288= 4320\cdot 672 = 30240\cdot 96$.
\end{example}
\subsection{An explicit non-trivial homology class
for rank $N=5$}\label{explicitclass}
The kernel of the $6\times 7$-matrix of $d_{10}$ for $\GL_5(\Z)$, displayed in the proof of Theorem \ref{splittingthm}, equation \eqref{diff10}
below, is spanned by $(0,0,0,0,0,1,-1)$ together with 
$(5,1,-8,16,15,2,2)$. The latter one provides a non-trivial homology class  in $H_{10}\big(\Vor_{GL_5(\Z)}\big)\iso H^{5}({GL_5(\Z)},\Z)$ (modulo $\cS_5$), given as a linear combination of cells (in terms of 
minimal vector indices) as follows:

\begin{eqnarray*}
&5&\varphi \big([1, 15, 4, 16, 10, 11, 17, 18, 3, 5]\big) \\
&+ &\varphi\big( [1, 15, 4, 16, 10, 11, 17, 18, 2, 5]\big)\\
&-8&\varphi \big( [1, 15, 4, 10, 11, 17, 18, 3, 2, 5]\big)\\
&+16& \varphi \big([1, 6, 15, 4, 16, 10, 11, 17, 18, 2]\big)\\
&+15&\varphi\big(  [1, 6, 15, 4, 16, 10, 11, 17, 18, 5]\big)\\
&+2 &\varphi \big([1, 6, 7, 4, 16, 10, 11, 17, 2, 5]\big)\\
&+2 &\varphi \big([1, 6, 7, 19, 13, 20, 15, 10, 11, 2]\big)\,.
\end{eqnarray*}

Here the indices refer to the following order for the set $m(P_5^1)$ of minimal vectors

\setcounter{MaxMatrixCols}{20}
{\tiny 
\begin{gather*}
\begin{matrix}
1&2&3&4&5&6&7&8&9&10&11&12&13&14&15&16&17&18&19&20\\
\hline\\
1&0&1&0&1&1&0&0&0&1&0&0&1&1&0&0&1&0&0&0\cr
0&1&1&0&0&-1&1&1&0&0&1&0&0&0&0&1&0&0&1&0\cr
0&0&0&-1&0&0&0&0&0&0&1&-1&-1&0&0&0&1&1&-1&1\cr
0&0&0&-1&0&0&0&-1&1&1&0&0&0&-1&-1&1&0&0&0&-1\cr
0&-1&-1&1&-1&0&0&0&-1&-1&-1&0&0&0&0&-1&-1&-1&0&0\cr
\end{matrix}
\end{gather*}
}

\begin{comment}
\setcounter{MaxMatrixCols}{20}
{\tiny
\begin{gather*}
\begin{matrix}
v_1&v_2&v_3&v_4&v_5&v_6&v_7&v_8&v_9&v_{10}&v_{11}&v_{12}&v_{13}&
v_{14}&v_{15}&v_{16}&v_{17}&v_{18}&v_{19}&v_{20}\\
\hline\\
1& 1& 0& 0& 1& 0& 0& 0& 1& 0& 0& 0& 1& 0& 0& 1& 0& 1& 0& 1 \cr
0&-1& 1& 1& 0& 0& 0& 1& 0& 0& 0& 1& 0& 0& 1& 0& 0& 1& 1& 0 \cr
0& 0& 0&-1&-1& 1& 1& 0& 0& 0& 1& 0& 0& 0& 1& 1& 1& 0& 0& 0 \cr
0& 0& 0& 0& 0& 0&-1&-1&-1& 1& 1& 1& 1& 1& 0& 0& 0& 0& 0& 0 \cr
0& 0& 0& 0& 0& 0& 0& 0& 0& 0&-1&-1&-1&-1&-1&-1&-1&-1&-1&-1 \cr
\end{matrix}
\end{gather*}
}
\end{comment}

\begin{comment}
\begin{eqnarray*}
&& [1, 0, 0, 0, 0], 
   [1, -1, 0, 0, 0], 
   [0, 1, 0, 0, 0], 
   [0, 1, -1, 0, 0], 
   [1, 0, -1, 0, 0],\\
&& [0, 0, 1, 0, 0], 
   [0, 0, 1, -1, 0], 
   [0, 1, 0, -1, 0], 
   [1, 0, 0, -1, 0], 
   [0, 0, 0, 1, 0], \\
&& [0, 0, 1, 1, -1], 
   [0, 1, 0, 1, -1], 
   [1, 0, 0, 1, -1], 
   [0, 0, 0, 1, -1], 
   [0, 1, 1, 0, -1],\\
&& [1, 0, 1, 0, -1], 
   [0, 0, 1, 0, -1], 
   [1, 1, 0, 0, -1], 
   [0, 1, 0, 0, -1], 
   [1, 0, 0, 0, -1]\,
\end{eqnarray*}
\end{comment}
\noindent
and $\varphi(u)$ for a vector $u$ of indices is the convex hull
of the minimal vectors corresponding to those indices,  as in \S\ref{ssec1.2}.

\section{Splitting off the Vorono\"{\i} complex $\Vor_N$ from $\Vor_{N+1}$ for small $N$}
In this section, we will be concerned with $\Gamma=GL_N(\Z)$ only and we adapt the
notation $\Sigma_n(N)= \Sigma_n(GL_N(\Z))$  for the sets of representatives.

\subsection{Inflating well-rounded forms}\label{inherit}
Let $A$ be the symmetric matrix attached to a form $h$ in $C_N^*$.
Suppose the cell associated to $A$ is {\em well-rounded}, i.e., its set of minimal vectors $S=S(A)$ spans the underlying vector space $\R^N$. Then we can associate to it a form $\tilde{h}$ with matrix $\wtA = 
 \begin{pmatrix} A&0\\ 0&m(A) \end{pmatrix}$ in $C_{N+1}^*$, where $m(A)$ denotes the minimum positive value of  $A$ on $\Z^N$. The set $\wtS$ of minimal vectors of $\wtA$ contains the ones from $S$, each vector being extended by an $(N+1)$-th coordinate 0. Furthermore, $\wtS$ contains the additional minimal vectors $\pm e_{N+1} =\pm
 (0,\dots,0,1)$, and hence it spans $\R^{N+1}$, i.e., $\wtA$ is well-rounded as well.
In the following, we will call forms like $\wtA$ as well as their associated cells {\em inflated}.

The stabilizer of  $h$ in $GL_N(\Z)$  thereby embeds into the one of $\tilde{h}$ inside
$\GL_{N+1}(\Z)$ (at least modulo $\pm {\rm Id}$) under the usual stabilization map.

Note that, by iterating the same argument $r$ times, $A$ induces a well-rounded form also in 
$\Sigma_\bullet^\star(N+r)$ which, 
for $r\geq 2$, does not belong to $\Sigma_\bullet(N+r)$
since there is an
obvious orientation-reversing automorphism of the inflated form, given by the permutation which 
swaps the last two coordinates.

\subsection{The case $N=5$} \label{case5}

\begin{comment}
It turns out that 
\begin{enumerate}
\item[1)] all the cells in $\Vor_{GL_5(\Z)} $ obtained from $\Sigma_n({5})$ by the above embedding of
well-rounded forms indeed survive in $\Sigma_{n+1}{(6)}$ (i.e., there is no orientation-reversing 
automorphism for the inflated cells);
\item[2)] each inflated cell in $\Sigma_{n+1}{(6)}$ is incident only with other inflated cells;
\item[3)] the stabilizers of the corresponding cells in $\Sigma_n(5)$ and $\Sigma_{n+1}(6)$ 
agree, up to a direct factor~$\Z/2$, arising from $-\Id_6$;
\item[4)] for any pair $(\sigma, \tau)\in \Sigma_n(5)\times \Sigma_{n-1}(5)$, the incidence number, with sign, in the differential agrees with the one for $(\tilde\sigma,\tilde \tau)\in \Sigma_{n+1}(6)\times \Sigma_{n}(6)$.
\end{enumerate}

\end{comment}

\begin{thm}\label{splittingthm}
The complex $\Vor_{GL_5(\Z)}$ is isomorphic to a direct factor of $\Vor_{GL_6(\Z)}$, with degrees shifted by 1.
\end{thm}

\begin{proof} The Vorono\"{\i} complex of $\GL_5(\Z)$ can be represented by the following 
weighted graph with levels
\vskip 20pt
\hbox{\hskip 0pt$
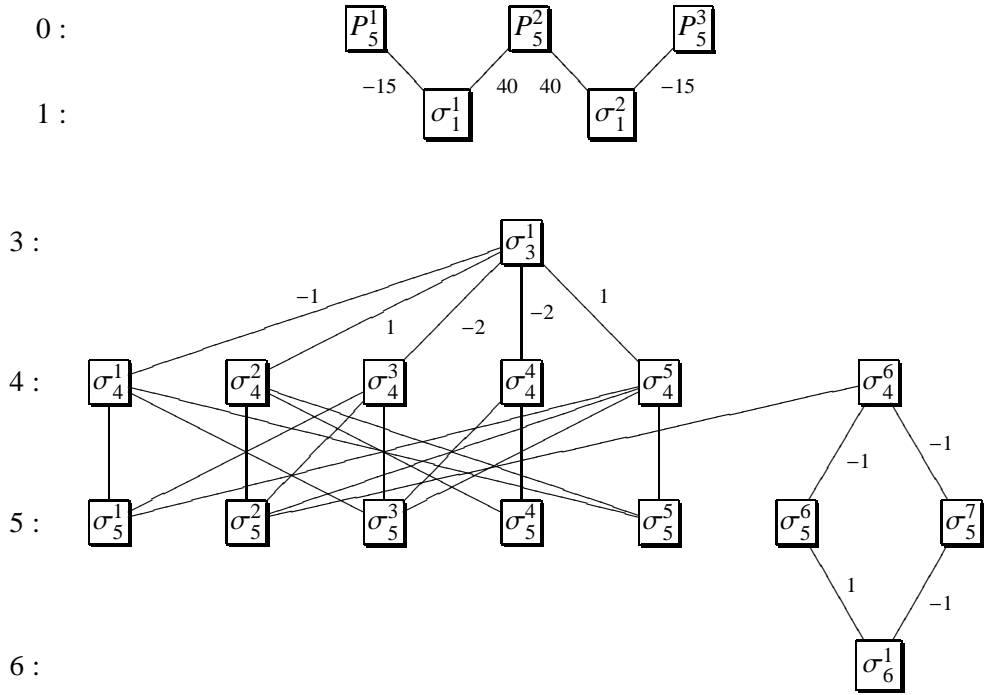
\xymatrix@=15pt{0:&&&&&
*+<3pt>[F-,]{P_{5}^1 }\ar@{-}[dr]_{-15}&&
*+<3pt>[F-,]{P_{5}^2 }\ar@{-}[dl]^{40}\ar@{-}[dr]_{40}&&
*+<3pt>[F-,]{P_{5}^3 }\ar@{-}[dl]^{-15}\\
1:& &&&&&
*+<5pt>[F-,]{\sigma_1^1 }&&
*+<5pt>[F-,]{\sigma_1^2} &}$
}

\vskip 30pt

\hbox{\hskip -10pt$\xymatrix@=15pt{3:&&&&&&&
*+<3pt>[F-,]{\sigma_{3}^{1} }\ar@{-}[ddllllll]_{-1}\ar@{-}[ddllll]^{1}\ar@{-}[ddll]^{-2}\ar@{-}[dd]^{-2}\ar@{-}[ddrr]^{1}\\ \\ 
4: &
*+<3pt>[F-,]{\sigma_{4}^{1} }\ar@{-}[dd]_{}\ar@{-}[ddrrrr]_{}\ar@{-}[ddrrrrrrrr]_{}  
&&
*+<3pt>[F-,]{\sigma_{4}^{2} }\ar@{-}[dd]_{}\ar@{-}[ddrrrr]_{{}}\ar@{-}[ddrrrrrr]_{{}}  
&&
*+<3pt>[F-,]{\sigma_{4}^{3} }\ar@{-}[ddllll]_{}\ar@{-}[ddll]_{}\ar@{-}[dd]_{}  
&&
*+<3pt>[F-,]{\sigma_{4}^{4} }\ar@{-}[dd]\ar@{-}[ddll]_{}   
&&
*+<3pt>[F-,]{\sigma_{4}^{5} }\ar@{-}[dd]^{} \ar@{-}[ddllll]_{} \ar@{-}[ddllllll]_{} \ar@{-}[ddllllllll]_{}  
&&&
*+<3pt>[F-,]{\sigma_{4}^{6} }\ar@{-}[ddlllllllll]_{} \ar@{-}[ddl]^{-1}\ar@{-}[ddr]^{-1}   %%433) [4,-1,-1]
\\  \\5: &
*+<3pt>[F-,]{\sigma_{5}^{1} }%\ar@{-}[ddrrrrrr]_{-2}    %%7
&&
*+<3pt>[F-,]{\sigma_{5}^{2} }%\ar@{-}[ddrrrr]_{1}  %% 2
&&
*+<3pt>[F-,]{\sigma_{5}^{3} }%\ar@{-}[ddrr]_{}  %%137
&&
*+<3pt>[F-,]{\sigma_{5}^{4} }%\ar@{-}[dd]_{}   %%82
&&
*+<3pt>[F-,]{\sigma_{5}^{5} }%\ar@{-}[ddll]_{}  %%84
&&
*+<3pt>[F-,]{\sigma_{5}^{6} }\ar@{-}[ddr]^{1}  %%160
&&
*+<3pt>[F-,]{\sigma_{5}^{7} }\ar@{-}[ddl]^{-1}   %%81
\\ \\
6:& &&&&
 &&&& &&&
*+<5pt>[F-,]{\sigma_{6}^{1} }&&& }$
}

\vskip 30pt

Here the nodes in line $j$ (marked on the left)  represent the elements in $\Sigma_{d(N)-j}(5)$, i.e.~we have 3, 2, 0, 1, 6, 7 and 1 cells in codimensions 0, 1, 2, 3, 4, 5 and 6, respectively, and arrows show incidences of those cells, while numbers attached to arrows give the corresponding incidence multiplicities. Since entering the multiplicities relating codimensions 4 and 5 would make the graph rather unwieldy, we give them instead in terms of the matrix corresponding to the differential $d_{10}$ 
connecting dimension 10 to 9 
(columns
refer, in this order, to  $\sigma_5^1,\dots,\sigma_5^7$, while rows refer to  $\sigma_4^1,\dots,\sigma_4^6$)
\begin{equation}\label{diff10}
\begin{pmatrix}
-5&0&-5&0&-1&0&0\\
0&-2&0&2&-2&0&0\\
2&-2&1&0&0&0&0\\
0&0&2&1&0&0&0\\
-1&-2&1&0&1&0&0\\
0&4&0&0&0&-1&-1
\end{pmatrix}\,.
\end{equation}

As is apparent from the picture, there are two connected components in that graph. The corresponding graph for $\GL_6(\Z)$ has three connected components, two of which are "isomorphic" (as weighted graphs with levels) to the one above for $\GL_5(\Z)$, except for a shift in codimension by 5 \big(e.g. codimension 0 cells in $\Sigma_\bullet(5)$
correspond to codimension~5 cells in $\Sigma_\bullet(6)$\big), i.e.~a shift in dimension by~1.

\medskip

In fact, it is possible, after appropriate coordinate transformations, to identify the minimal vectors (viewed up to sign)
of any given cell in the two inflated components of $\Sigma_\bullet{(6)}$ alluded to above with the minimal vectors of another cell which is inflated from one in $\Sigma_\bullet{(5)}$, except precisely {\em one} minimal vector (up to sign) which is {\em fixed} under the stabilizer of the cell.

\medskip
Let us illustrate this correspondence for the top-dimensional cell $\sigma$ of the
perfect form $P_{5}^{1}\in \Sigma_{14}(5)$, also denoted $P(5,1)$ in \cite{jaquet} and  $D_5$ in \cite{LS}, with the  
list $m(P_5^1)$ of minimal vectors given already at the end of \S\ref{explicitclass}.
\begin{comment}
\begin{eqnarray*}
&& (1, 0, 0, 0, 0),
 (0, 1, 0, 0, -1),
 (1, 1, 0, 0, -1),
 (0, 0, -1, -1, 1),
 (1, 0, 0, 0, -1), \\
&& (1, -1, 0, 0, 0),
 (0, 1, 0, 0, 0),
 (0, 1, 0, -1, 0),
 (0, 0, 0, 1, -1),
 (1, 0, 0, 1, -1),\\
&& (0, 1, 1, 0, -1),
 (0, 0, -1, 0, 0),
 (1, 0, -1, 0, 0),
 (1, 0, 0, -1, 0),
 (0, 0, 0, -1, 0),\\
&& (0, 1, 0, 1, -1),
 (1, 0, 1, 0, -1),
 (0, 0, 1, 0, -1),
 (0, 1, -1, 0, 0),
 (0, 0, 1, -1, 0).
\end{eqnarray*}
\end{comment}

\begin{comment}
\setcounter{MaxMatrixCols}{21}
{\tiny 
\begin{gather*}
\begin{matrix}
1&2&3&4&5&6&7&8&9&10&11&12&13&14&15&16&17&18&19&20\\
\hline\\
1&0&1&0&1&1&0&0&0&1&0&0&1&1&0&0&1&0&0&0\cr
0&1&1&0&0&-1&1&1&0&0&1&0&0&0&0&1&0&0&1&0\cr
0&0&0&-1&0&0&0&0&0&0&1&-1&-1&0&0&0&1&1&-1&1\cr
0&0&0&-1&0&0&0&-1&1&1&0&0&0&-1&-1&1&0&0&0&-1\cr
0&-1&-1&1&-1&0&0&0&-1&-1&-1&0&0&0&0&-1&-1&-1&0&0\cr
\end{matrix}
\end{gather*}
}
\end{comment}

\begin{comment}
&& (0, 0, 0, -1, 0),
 (0, 0, -1, 0, 0),
 (1, 1, 0, 0, -1),
 (0, 0, -1, -1, 1),
 (1, 0, 0, 0, -1), \\
&& (0, 1, -1, 0, 0),
 (0, 1, 0, -1, 0),
 (1, 0, 0, -1, 0),
 (1, 0, -1, 0, 0),
 (0, 1, 0, 0, -1),\\
&& (0, 0, 1, -1, 0),
 (0, 0, 0, 1, -1),
 (0, 0, 1, 0, -1),
 (1, 0, 1, 0, -1),
 (1, 0, 0, 1, -1),\\
&& (0, 1, 0, 0, 0),
 (1, 0, 0, 0, 0),
 (0, 1, 0, 1, -1),
 (0, 1, 1, 0, -1),
 (1, -1, 0, 0, 0)\,.
 \end{comment}
 
Using the algorithm described in \S\ref{equiv_cells}, the corresponding inflated cell $\widetilde\sigma$ in $\Sigma_{15}{(6)}$ can be found to be, in terms of its 21 minimal vectors of the perfect form $P_6^1$ 
in Jaquet's notation (see \cite{jaquet} and \S\ref{explicitclass} for the full list $m(P_6^1)$),

\begin{comment}
$$\begin{tabular}{l c lrr r r r r}
  $v_1  $ &=&  $\pm$ (& 1 &   0 &   0 &   0  &  0  &  0 )\,,\\ 
   $v_2  $ &=&  $\pm$ (&$-1$ &  1 &  0 &  0 &  0 &  0 )\,,\\
   $v_4  $ &=&  $\pm$ (&  0 &$-1$ &  1 &  0 &  0 &  0 )\,,\\
   $v_5  $ &=&  $\pm$ (&$-1$ &  0 &  1 &  0 &  0 &  0 )\,,\\
   $v_{10}$  &=&  $\pm$ (&  0  &  0 &  0 &  1 &  0 &  0 )\,,\\
   $v_{12} $ &=&  $\pm$ (&  0 &  0 &$-1$ &  0 &  1 &  0 )\,,\\
   $v_{13} $ &=&  $\pm$ (&  0  &$-1$ &  0 &  0 &  1 &  0 )\,,\\
   $v_{14} $ &=&  $\pm$ (&$-1$ &  0 &  0 &  0 &  1 &  0 )\,,\\
   $v_{16} $ &=&  $\pm$ (&  1 &  1 &$-1$ &$-1$ &$-1$ &  1 )\,,\\
   $v_{17} $ &=&  $\pm$ (&  1 &  0 &  0 &$-1$ &$-1$ &  1 )\,,\\
   $v_{18} $ &=&  $\pm$ (&  0 &  1 &  0 &$-1$ &$-1$ &  1 )\,,\\
   $v_{22} $ &=&  $\pm$ (&  1 &  1 &$-1$ &  0 &$-1$ &  1 )\,,\\
  $v_{24} $ &=&  $\pm$ (&  1 &  0 &  0 &  0 &$-1$ &  1 )\,,\\
  $v_{25} $ &=&  $\pm$ (&  0 &  1 &  0 &  0 &$-1$ &  1 )\,,\\
  $v_{26} $ &=&  $\pm$ (&  1 &  0 &$-1$ &$-1$ &  0 &  1 )\,,\\
  $v_{27} $ &=&  $\pm$ (&  0 &  1 &$-1$ &$-1$ &  0 &  1 )\,,\\
  $v_{29} $ &=&  $\pm$ (&  0 &  0 &  0 &$-1$ &  0 &  1 )\,,\\
  $v_{33} $ &=&  $\pm$ (&  1 &  0 &$-1$ &  0 &  0 &  1 )\,,\\
  $v_{34} $ &=&  $\pm$ (&  0 &  1 &$-1$ &  0 &  0 &  1 )\,,\\
  $v_{35} $ &=&  $\pm$ (&  0 &  0 &  0 &  0 &  0 &  1 )\,,\\
  $v_{36} $ &=&  $\pm$ (&  1 &  1 &$-1$ &$-1$ &$-1$ &  2 )\,.
\end{tabular}
$$
\end{comment}

\setcounter{MaxMatrixCols}{21}
{\tiny
\begin{gather*}
\begin{matrix}
v_1&v_2&v_4&v_5&v_{10}&v_{12}&v_{13}&v_{14}&v_{16}&v_{17}&v_{18}&v_{22}&v_{24}&v_{25}&v_{26}
&v_{27}&v_{29}&v_{33}&v_{34}&v_{35}&v_{36}\\
\hline\\
1&-1&0&-1&0&0&0&-1&1&1&0&1&1&0&1&0&0&1&0&0&1\cr
0&1&-1&0&0&0&-1&0&1&0&1&1&0&1&0&1&0&0&1&0&1\cr
0&0&1&1&0&-1&0&0&-1&0&0&-1&0&0&-1&-1&0&-1&-1&0&-1\cr
0&0&0&0&1&0&0&0&-1&-1&-1&0&0&0&-1&-1&-1&0&0&0&-1\cr
0&0&0&0&0&1&1&1&-1&-1&-1&-1&-1&-1&0&0&0&0&0&0&-1\cr
0&0&0&0&0&0&0&0&1&1&1&1&1&1&1&1&1&1&1&1&2\cr
\end{matrix}
\end{gather*}
}\noindent

The transformation
$$\gamma = \begin{pmatrix} 
0&-1&-1&0&0&0\\
0&0&-1&0&-1&-1\\
0&0&0&1&0&1\\
0&0&0&1&0&0\\
0&0&1&-1&0&0\\
-1&-1&-1&0&-1&0
\end{pmatrix}$$
\begin{comment}
1&1&1&0&1&-1\\
-1&0&0&-1&-1&0\\
0&0&0&1&1&0\\
0&0&0&1&0&0\\
0&-1&-1&0&-1&0\\
0&0&1&-1&0&0

 1&0&0&1&0&0 \\
 0&0&0&1&0&0 \\
 0&0&0&1&0&1 \\
 0&0&1&0&1&0 \\
 0&1&0&0&1&0 \\
 0&0&0&-1&-1&-1
\end{comment}
sends $v_1$ to $(0,0,0,0,0,1)$ and sends each of the other vectors to the corresponding one of the form $(v,0)$ where $v$ is  the corresponding minimal
vector for $P_5^1$ (in the order given above).

One can verify that the other two perfect forms $P_{5}^{2}$ and
$P_{5}^{3}$ (denoted by Vorono\"{i} $A_5$ and $\varphi_2 $, respectively) give rise to a corresponding inflated cell in $\Sigma_{15}{(6)}$ in a similar way.

Concerning  the cells of {\em positive} codimension in  $\Sigma_\bullet{(5)}$, it turns out that these all have a representative which is a facet in $\sigma$. Furthermore, the matrix $\gamma$ induces an isomorphism from the subcomplex of $\Sigma_\bullet(6)$ spanned by $\widetilde\sigma$ and all its facets to the complex obtained by inflation, as in \S\ref{inherit} above, from the complex spanned by $\sigma_5$ and all its facets. Finally, one can verify that the cells attached to $P_5^2$ and $P_5^3$ are conjugate, after inflation, to cells in $\Sigma_{15}(6)$, and that the differentials for $\Vor_{\GL_5}$ and   $\Vor_{\GL_6}$ agree on these.  This ends the proof of the theorem.  \end{proof}  

\medskip
\subsection{Other cases}
\label{case3}
A similar situation holds for $\Sigma_\bullet{(3)}$ and $\Sigma_\bullet{(4)}$, but as $\Sigma_\bullet{(3)}$ consists of 
a single cell only, the picture is far less significant.

For $N=4$, there is only one cell leftover in $\Sigma_\bullet(4)$, in fact in $\Sigma_6(4)$, and it is already
inflated from $\Sigma_5(3)$, as shown in \S\ref{case3}. Hence its image in $\Sigma^\star_7(5)$ will allow an orientation reversing
automorphism and hence will not show up in $\Sigma_7(5)$. This illustrates the remark at the end of
\ref{inherit}.

Finally, for $N=6$, the cells in the third component of the incidence graph for $\GL_6(\Z)$ mentioned in the proof of Theorem \ref{splittingthm} above appear, in inflated form,
in the Vorono\"{\i} complex for $\GL_7(\Z)$ which inherits the homology of that component, 
since in the weighted graph of $GL_7(\Z)$, which is connected, there is only one incidence of an 
inflated cell with a non-inflated one.
Therefore we do not have a splitting in this case.

\section{The Cohomology of modular groups}\label{sec4}
\subsection{Preliminaries}
Recall the following simple fact:

\begin{lem}\label{lemma1}
Assume that $p$ is a prime and $g \in {\rm GL}_N ({\mathbb R})$ has order $p$. Then $p \leq N+1$.
\end{lem}
\begin{proof} The minimal polynomial of $g$ is the cyclotomic polynomial $x^{p-1} + x^{p-2} + \cdots + 1$. By the Cayley-Hamilton theorem, this polynomial divides the characteristic polynomial of $g$. Therefore $p-1 \leq N$. \end{proof}

We shall also need the following result:

\begin{lem}\label{lemma2}
The action of ${\rm GL}_N ({\mathbb R})$ on the symmetric space $X_N$ preserves its orientation if and only if $N$ is odd.
\end{lem}

\begin{proof} The subgroup $ {\rm GL}_N ({\mathbb R})^+ \subset  {\rm GL}_N ({\mathbb R})$ of elements with positive determinant is the connected component of the identity, therefore it preserves the orientation of $X_N$. Any $g \in  {\rm GL}_N ({\mathbb R})$ which is not in $ {\rm GL}_N ({\mathbb R})^+$ is the product of an element of ${\rm GL}_N ({\mathbb R})^+$ with the diagonal matrix $\varepsilon = {\rm diag} \, (-1,1,\ldots , 1)$, so we just need to check when $\varepsilon$ preserves the orientation of $X_N$. The tangent space $TX_N$ of $X_N$ at the origin consists of real symmetric matrices $m = (m_{ij})$ of trace zero. The action of $\varepsilon$ is given by $m \cdot \varepsilon = \varepsilon^t \, m \, \varepsilon$ (cf. \S\ref{ssec1.1}) and we get
$$
(m \cdot \varepsilon)_{ij} = m_{ij}
$$
unless $i=1$ or $j=1$ and $i \ne j$, in which case $(m \cdot \varepsilon)_{ij} = -m_{ij}$. Let $\delta_{ij}$ be the matrix with entry $1$ in row $i$ and column $j$, and zero elsewhere. A basis of $TX_N$ consists of the matrices $\delta_{ij} + \delta_{ji}$, $i \ne j$, together with $N-1$ diagonal matrices. For this basis, the action of $\varepsilon$ maps $N-1$ vectors $v$ to their opposite $-v$ and fixes the other ones. The lemma follows. \end{proof}

\subsection{Borel--Serre duality}\label{Farrellcohomology}

According to Borel and Serre (\cite{BS}, Thm.~11.4.4 and Thm.~11.5.1), the group $\Gamma=\SL_N(\Z)$ or $GL_N(\Z)$ is a virtual duality group with dualizing module
$$
H^{v(N)} (\Gamma , {\mathbb Z} [\Gamma]) = {\rm St}_N \otimes \tilde{\mathbb Z} \, ,
$$
where $v(N) = N(N-1)/2$ is the virtual cohomological dimension of $\Gamma$ and $\tilde{\mathbb Z}$ is the orientation module of $X_N$. It follows that there is a long exact sequence
\begin{equation}\label{lesbos}
\cdots \to H_n (\Gamma , {\rm St}_N) \to H^{v(N)-n} (\Gamma , \tilde{\mathbb Z}) \to \hat H^{v(N)-n} (\Gamma , \tilde{\mathbb Z}) \to H_{n-1} (\Gamma , {\rm St}_N) \to \cdots
\end{equation}
where $\hat H^*$ is the Farrell cohomology of $\Gamma$ \cite{FA}. From Lemma \ref{lemma1} and the Brown spectral sequence (\cite{B}, X (4.1)) we deduce that $\hat H^* (\Gamma , \tilde{\mathbb Z})$ lies in ${\mathcal S}_{N+1}$. Therefore
\begin{equation}\label{BorelSerre}
H_n (\Gamma , {\rm St}_N) \equiv  H^{v(N)-n} (\Gamma , \tilde{\mathbb Z}) \, , \ \mbox{modulo ${\mathcal S}_{N+1}$.}
\end{equation}
When $N$ is odd, then ${\rm GL}_N ({\mathbb Z})$ is the product of ${\rm SL}_N ({\mathbb Z})$ by ${\mathbb Z} / 2$, therefore
$$
H^m ({\rm GL}_N ({\mathbb Z}) , {\mathbb Z}) \equiv H^m ({\rm SL}_N ({\mathbb Z}) , {\mathbb Z}) \, , \ \mbox{modulo ${\mathcal S}_2$.}
$$
When $N$ is even, then the action of ${\rm GL}_N ({\mathbb Z})$ on $\tilde{\mathbb Z}$ is given by the sign of the determinant (see Lemma \ref{lemma2}) and Shapiro's lemma gives
\begin{equation}\label{Shapiro}
H^m ({\rm SL}_N ({\mathbb Z}) , {\mathbb Z}) = H^m ({\rm GL}_N ({\mathbb Z}) , M) \, ,
\end{equation}
with 
$$
M = {\rm Ind}_{{\rm SL}_N ({\mathbb Z})}^{{\rm GL}_N ({\mathbb Z})} \, {\mathbb Z} \equiv {\mathbb Z} \oplus \tilde {\mathbb Z} \, , \ \mbox{modulo ${\mathcal S}_2$.}
$$
To summarize: when $\Gamma=\SL_N(\Z)$ or $GL_N(\Z)$, where $N\leq 7$, we know $H^m(\Gamma,\tilde{\mathbb Z})$ by 
combining \eqref{eq3} (end of \S\ref{ssec2.4}), Theorem \ref{homology_voronoi} and \eqref{BorelSerre}. This allows us to compute the 
cohomology of $GL_N(\Z)$. The results are given in Theorem \ref{cohomology} below.

\subsection{The cohomology of modular groups}
\begin{thm}\label{cohomology}
 \begin{itemize}
\item[(i)] Modulo ${\mathcal S}_5$ we have
$$
H^m ({\SL}_5 (\Z) , \Z) = \begin{cases} \Z & {\rm if} \quad m = 0, 5,\ \\
0 & {\rm otherwise}.
\end{cases}
$$
\item[(ii)] Modulo ${\mathcal S}_7$ we have
$$
H^m ({GL}_6 (\Z) , \Z) = \begin{cases} \Z &{\rm if }\quad m= 0, 5, 8,\\
0 & {\rm otherwise},
\end{cases}
$$
and
$$
H^m ({\SL}_6 (\Z) , \Z) = \begin{cases} \Z^2 &{\rm if }\quad m=5,\\
\Z &{\rm if }\quad m=0, 8, 9, 10,\ \\
0 & {\rm otherwise}.
\end{cases}
$$
\item[(iii)] For $N=7$, we have,
$$
H^m ({\SL}_7 (\Z) , \Q) = \begin{cases} \Q &{\rm if }\quad m=0, 5, 11, 14, 15,\\
0 & {\rm otherwise}.
\end{cases}
$$
\end{itemize}
\end{thm}

\begin{rem} Morita asks in \cite{Morita} whether the class of infinite
order in $H^5(GL_5(\Z),\Z)$ survives in the cohomology
of the group of outer automorphisms of the free group
of rank five.
\end{rem}

\def \ft{\rm}


\begin{thebibliography}{60}

\bibitem{batut} {\bfseries Batut, C.;} {\em Classification of quintic eutactic forms,} {\ft Math. Comput.} {\bf 70}, no.233 (2001), 395--417. 
\bibitem{barnes} {\bfseries Barnes, E.S.;} {\em  The complete enumeration of extreme senary forms,} 
{\ft Phil. Trans. Roy. Soc. London} {\bf249-A} (1957), 461--506.
\bibitem{Bavard} {\bfseries Bavard, C.;} {\em  Classes minimales de r\'eseaux et r\'etractions g\'eom\'etriques \'equivariantes dans les espaces sym\'etriques}, J. London Math. Soc. (2) {\bf 64}, no. 2 (2001), 275--286.
\bibitem{B} {\bfseries Brown, K.;} {\em Cohomology of Groups,} Graduate Texts in Mathematics {\bf87}, Springer-Verlag, New York, 1994. 
\bibitem{BS} {\bfseries Borel, A.; Serre, J--P.;} {\em Corners and arithmetic groups,} 
Comment. Math. Helv.  {\bf48}  (1973), 436--491.
\bibitem{pasco2007} {\bfseries Dumas, J.--G.; Elbaz--Vincent, Ph.; Giorgi, P.; Urbanska, A.;} {\em Parallel Computation of the Rank of Large Sparse Matrices from Algebraic K-theory.},  PASCO 2007: Parallel Symbolic Computation '07, 26--27 July, Waterloo, Canada.  Proceedings of the ACM (2007), 43--52.
\bibitem{dsv} {\bfseries Dutour Sikiric, M.; Sch\"urmann, A.; Vallentin, F.;} {\em Classification of eight dimensional perfect forms,} {\ft Electron. Res. Announc. Amer. Math. Soc.} {\bf 13} (2007), 21--32.
\bibitem{FA} {\bfseries Farrell, F. T.;} {\em An extension of Tate cohomology to a class of infinite groups,}
{\ft J. Pure Appl. Algebra} {\bf 10}, no. 2 (1977/78), 153--161. 
\bibitem{GAP} {\bfseries The GAP Group;} GAP -- Groups, Algorithms, and Programming, Version {\tt 4.4.12}; 2008. (\url{http://www.gap-system.org}).
\bibitem{harder} {\bfseries Harder, G.;} {\em Die Kohomologie $S$-arithmetischer Gruppen \"uber Funktionenk\"orpern,} {\ft  Invent. Math.} {\bf 42}  (1977), 135--175. 
\bibitem{jaquet} {\bfseries Jaquet, D.-O.;} {\em \'Enum\'eration compl\`ete des classes de 
formes parfaites en dimension 7}, Th\`ese de doctorat, Universit\'e de Neuch\^atel
(1991).
\bibitem{LS}  {\bfseries Lee, R.; Szczarba, R. H.;} {\em On the torsion in $K_4(\Z)$ and $K_5(\Z)$}, {\ft Duke Math. J.} {\bf 45} (1978), 101--129.
\bibitem{martinet} {\bfseries Martinet, J.;} {\em  Perfect Lattices in Euclidean Spaces,}  Springer-Verlag, Grundlehren der Mathematischen Wissenschaften {\bf 327}, Heidelberg, 2003.
\bibitem{Morita}  {\bfseries Morita, S.;} {\em Cohomological structure of the mapping class group and beyond.} in: Problems on mapping class groups and related topics,
Proc. Sympos. Pure Math., {\bf 74}, Amer. Math. Soc., Providence, RI, 2006, 329--354.
 \bibitem{PARI2} {\bfseries The PARI--Group;} PARI/GP, versions {\tt 2.1--2.4}, Bordeaux, \url{http://pari.math.u-bordeaux.fr/}.
\bibitem{PFPK} {\bfseries PFPK:} A C library for computing Voronoï complexes, version {\tt 1.0.0}, 2009. 
\bibitem{souvignier} {\bf Plesken, W.; Souvignier, B.;} {Computing isometries 
of lattices}, {\ft J. Symb. Comput.} {\bf24} (1997), 327--334. 
\bibitem{Sch} {\bfseries Sch\"urmann, A.;} {\em Enumerating perfect forms.} arXiv:0901.1587 (math.NT).
\bibitem{Serre} {\bfseries Serre, J--P.;} {\em Trees}, Springer Monographs in Mathematics. Springer-Verlag, Berlin, 2003.
\bibitem{SouleK4} {\bfseries Soul\'e, C.;} {\em On the $3$-torsion in $K\sb 4(Z)$.} {\ft Topology} {\bf 39}, no.2, (2000),  259--265.
\bibitem{Soule-SL3} {\bfseries Soul\'e, C.;} {\em The cohomology of  $\SL\sb 3(Z)$.} {\ft Topology} {\bf 17}, no. 1 (1978), 1--22. 
\bibitem{Vo}{\bfseries  Vorono{\"\i}, G.:} {\em Nouvelles applications des param\`etres continus \`a la th\'eorie des formes quadratiques I}, {\ft Crelle} 133 (1907), 97--178.
\end{thebibliography}
\end{document}